\documentclass[a4paper,11pt]{article}
\usepackage{amsmath,amsthm,amssymb}

\newcommand{\refpart}[1]{{\it (#1)}}  

\renewcommand\ge\geqslant
\renewcommand\geq\geqslant
\renewcommand\le\leqslant
\renewcommand\leq\leqslant

\newtheorem{theorem}{Theorem}[section]

\newtheorem{conjecture}[theorem]{Conjecture}
\newtheorem{example}[theorem]{Example}

\newtheorem{observe}[theorem]{Observation}
\newtheorem{remark}[theorem]{Remark}
\newtheorem{definition}[theorem]{Definition}
\newtheorem{notation}[theorem]{Notation}

\newcommand{\CC}{\mathbb{C}}
\newcommand{\RR}{\mathbb{R}}
\newcommand{\QQ}{\mathbb{Q}}

\newcommand{\PP}{\mathbb{P}}

\newcommand{\cH}{{\mathcal H}}

\newcommand{\OP}{{\mathcal L}}

\newcommand{\w}{w} %\gamma}
\newcommand{\q}{q}
\newcommand{\s}{s}
\newcommand{\V}{V}

\newcommand{\hpg}[5]{{}_{#1}\mbox{\rm F}_{\!#2}\!
  \left(\left.{#3 \atop #4}\right| #5 \right) }

\newcommand{\frc}[2]{\stackrel{#1}{}\hspace{-3.5pt}/#2}
\title{Differential relations for almost Belyi maps}

\author{
        Raimundas Vidunas\footnote{% Graduate School of Information Science and Technology, 
        Osaka University, Osaka, Japan.
        E-mail: {\sf rvidunas@gmail.com}.}
        \hspace*{1cm}
        Jiro Sekiguchi\footnote{
        Tokyo University of Agriculture and Technology. 
        E-mail: {\sf sekiguti@cc.tuat.ac.jp }.}
       }
      
\begin{document}

\date{}
\maketitle    

\begin{abstract}
Several kinds of differential relations for polynomial components
of almost Belyi maps are presented. 
Saito's theory of free divisors give particularly interesting 
(yet conjectural) logarithmic action of vector fields. 
% (in the variable and the parameter)
The differential relations implied by Kitaev's construction 
of algebraic Painlev\'e VI solutions through pull-back transformations
% of hypergeometric equations to isomonodromic Fuchsian equations 
% corresponding to algebaic Painlev\'e VI solutions 
% (by the Jimbo-Miwa correspondence and Kitaev's construction) 
are used to compute almost Belyi maps for the pull-backs giving all genus 0 and 1 
Painlev\'e VI solutions in the Lisovyy-Tykhyy classification.
\end{abstract}

\section{Introduction}

Importance of Belyi maps was highlighted in the
{\it l'Esquisse d'une programme} by Grothendieck \cite{Groth84}.
Since then, Belyi maps attract increasing attention in algebraic geometry,
number theory, mathematical physics. One elementary application of
Belyi maps is pull-back transformations of hypergeometric differential equations
to Fuchsian equations with a small number of singularities, and 
corresponding transformations of special functions \cite{vHVHeun}, \cite{vhpg}, \cite{vf}.

Recall that a {\em Belyi map} % (or a  a {\em Belyi function}) 
is an algebraic covering $\varphi:C\to\PP^1$
% of the projective line  on an algebraic curve 
that branches only above $\{0,1,\infty\}\subset\PP^1$.
In particular, a genus 0 Belyi map (with $C\cong\PP^1$)
is defined by a rational function 
$\varphi(x)\in\CC(x)$ such that all branching points $\{x:\varphi'(x)=0\}$ 
lie in the fibers $\varphi(x)\in \{0,1,\infty\}$.

Almost Belyi maps were assertively introduced by Kitaev \cite{Kit1}, \cite{Kit2} 
in the context of algebraic solutions of the Painlev\'e VI equation.
 \begin{definition} \rm
An {\em almost Belyi map} (or an {\em AB-map}, for shorthand)
is an algebraic covering $\varphi:C\to\PP^1$
that has exactly one simple branching point outside the fibers  % $\varphi\in
$\{0,1,\infty\}\subset\PP^1$. (Recall that {\em simple} branching points have 
the branching order 2.)
\end{definition}
Kitaev constructed algebraic Painlev\'e VI functions 
using the Jimbo-Miwa correspondence \cite{JM81} 
to isomonodromic $2\times 2$ Fuchsian systems with 4 singularities. 
The corresponding Fuchsian systems are generated by pull-backs % transformations
of the Gauss-Euler hypergeometric equation with respect to AB-maps.
In the context of Picard-Fuchs equations, the same pull-back method 
with AB-maps was employed by Doran \cite{Doran}, Movasatti, Reiter \cite{MR}.

Recently \cite{kms15}, algebraic Painlev\'e VI solutions and AB-maps found application 
in Saito's singularity theory \cite{Saito}, \cite{Se1}
and Dubrovin's % more general 
theory of Frobenius manifolds \cite{Dubr}.
This direction  motivates computation of new examples of AB-maps.
In particular, a list of AB-maps giving pull-backs to all cases of 
algebraic Painlev\'e VI solutions in the Lisovyy-Tykhyy classification \cite{LiTy}
(up to Schlessinger gauge transformations) is desirable.

The problem of computing AB-maps and the mentioned applications
give an interesting set of differential relations for AB-maps 
and their polynomial components. In particular, 
\begin{itemize}
\item Usefulness of differentiation in computing Belyi maps was noticed
by several authors \cite[\S 2.5]{sv}. Further, the very fact of implied 
pull-back transformations of Fuchsian equations
gives additional differential and algebraic restrictions.
The same techniques apply to computation of AB-maps,
as we demonstrate in \S \ref{sec:compute}.
\item Kitaev's basic construction entails differentiation 
with respect to the ``isomondromic" {\em parameter} 
(rather than with respect to the independent variable),
leading to differential relations between the coefficients of an AB-map. 
The straightforward case of Kitaev's {\em RS-transformations}
is summarized in Theorem \ref{th:kitaev}.
\item Saito's construction of {\em free divisors} gives action of vector fields
that relates differentiation both with respect to the independent variable 
and the  ``iso-mondromic" {\em parameter}. Remarkably, we observe 
existence of vector fields that are {\em logarithmic} along each hypersurface
defined by the polynomial components of an AB-map,
% that the action of the vector fields is {\em logarithmic} even on each polynomial component, 
leading us to Conjecture \ref{th:conj}.
\end{itemize}
Analysis of these differential relations 
(in \S \ref{sec:compute}, \S \ref{sec:relap6}, \S \ref{sec:freediv}, respectively)
is the main contribution of this article. 
Additionally, Section \ref{sec:algp6} presents computational results 
of AB-maps for all genus 0 and 1 cases of the Lisovyy-Tykhyy classification \cite{LiTy} 
of algebraic Painlev\'e VI solutions. 

\section{Preliminaries}

Here we introduce % Belyi and almost Belyi functions (of genus 0);
application of Belyi maps and AB-maps 
to pull-back transformations between Fuchsian equations; % with a few singularities; 
basic methods for computing these maps and differential relations they employ.

% At the Chern Institute conference in July, Shabat proposed 
% to call almost Belyi maps as Fried maps.

\subsection{Nomenclature for AB-maps}
\label{sec:nomenc}

This paper studies AB-maps of genus 0.
We are thus looking at rational functions $\varphi(x)\in\CC(x)$
such that all branching points $\{x:\varphi'(x)=0\}$ {\em except one}
lie in the fibers $\varphi(x)\in \{0,1,\infty\}$.
The extra branching point has the branching order 2
(thus $\varphi''\neq 0$ at that branching point if it is not $\infty$).

Important distinctions between Belyi maps and AB-maps are:
\begin{enumerate}
\item Belyi maps form discrete (0-dimensional) Hurwitz spaces.
AB-maps form 1-dimensional Hurwitz spaces; 
that is, there are 1-dimensional families of them 
parametrized by algebraic curves.
\item By Hurwitz theorem, a Belyi map $\varphi:\PP^1\to\PP^1$ 
of genus 0, degree $d$
has exactly $d+2$ distinct points in the 3 fibers $\varphi(x)\in \{0,1,\infty\}$. 
An AB-map $\varphi:\PP^1\to\PP^1$ of genus 0, degree $d$ 
has exactly $d+3$ points in the 3 fibers.
\end{enumerate}
% Kitaev \cite{Kit1}, \cite{Kit2} originated systematic study of 
% almost Belyi maps for investigating algebraic Painlev\'e VI
% solutions and isomonodromic Fuchsian equations.
\begin{example} \rm \label{eq:d6example}
An example of a AB-map of degree 6 is 
\begin{equation} \label{eq:ab6}
\varphi_1(x)=\frac{(\w\,x^3+15x^2+20x+8)^2}{64\,(x+1)^5}.
\end{equation}
The parameter $\w$ appears only once.
We can compute:
\begin{align*}
\varphi_1(x)-1&=
\frac{x^3\big(\w^2x^3+2\,(15\w-32)x^2+5\,(8\w-19)x+16\w-40\big)}{64\,(x+1)^5},\\
\varphi_1'(x) &=\frac{x^2\,(\w\,x^3+15x^2+20x+8)\,(\w\,x+6\w-15)}{64\,(x+1)^5}.
\end{align*}
The root $\,x=\q_1=-6+15/\w\,$ of $(\w\,x+6\w-15)$ is the only branching point 
outside the fibers $\varphi(x)\in \{0,1,\infty\}$.
\end{example} 

\begin{notation} \rm
% By {\em exceptional fibers} of a rational function $\varphi\in\CC(x)$
% we mean the three fibers $\varphi=0$, $\varphi=1$, $\varphi=\infty$.
Let $\varphi\in\CC(x)$ be a rational function of degree $d$.
The {\em branching pattern} in a fiber $\varphi=C$ is given by a partition of $d$.
We choose the multiplicative notation $1^{n_1}2^{n_2}\ldots$ for a branching pattern,
meaning $n_1$ non-branching points, $n_2$ branching points of order 2, etc.
For example,  we write the branching pattern of the  fiber $\varphi_1=1$ 
of the AB-map in (\ref{eq:ab6})  as $1^33$ rather than $1+1+1+3$. 
The partition fact is expressed by $\sum kn_k=d$. 

The collection $[P_1/P_2/P_3]$ of the branching patterns $P_1$, $P_2$, $P_3$
in the fibers $\varphi=0$, $\varphi=1$, $\varphi=\infty$ is called at 
the {\em passport} of $\varphi$. For example, the passport of $\varphi_1$ 
in (\ref{eq:ab6}) is $[2^3/3\,1^3/5\,1]$, keeping in mind the point $x=\infty$
in the fiber $\varphi=\infty$. The order of branching patterns in the passport
is not significant to us, as permutation of the 3 fibers is realized by 
the fractional-linear expressions $\varphi/(\varphi-1)$, $1-\varphi$, $1/\varphi$,
$1/(1-\varphi)$, $(\varphi-1)/\varphi$.
\end{notation}

\subsection{Pull-backs of Fuchsian equations}

One application of Belyi maps is pull-back transformations of  the hypergeometric %differential 
equation
\begin{align} \label{eq:hpge}
\frac{d^2y(z)}{dz^2}+
\left(\frac{\,c\,}{z}+\frac{a+b-c+1}{z-1}\right) \frac{dy(z)}{dz}+\frac{a\,b}{z\,(z-1)}\,y(z)=0.
\end{align} 
to Fuchsian equations with a few singularities (e.g., Heun, other hypergeometric equations).
The pull-back transformations have the form 
\begin{equation} \label{eq:pback}
z\longmapsto\varphi(x), \qquad
y(z)\longmapsto Y(x)=\theta(x)\,y(\varphi(x)),
\end{equation}
where $\varphi(x)$ is a rational function, % often a Belyi function,
and $\theta(x)$ is a Liouvillian (e.g., power) function.
The rational function $\varphi(x)$ is typically a special Belyi map.
Applicable Belyi maps are characterized using the following definition 
\cite[Definition 1.2]{vHVHeun}.
\begin{definition} \label{df:klmnregular} \rm
Given positive integers $k,\ell,m,n$, a Belyi map $\varphi:\PP_x^1\to\PP_z^1$
is called {\em $(k,\ell,m)$-minus-$n$ regular} if, 
with exactly $n$ exceptions in total, 
all points above $z=1$ have branching order $k$, 
all points above $z=0$ have branching order $\ell$, and
all points above $z=\infty$ have branching order $m$. 
\end{definition} 

The singularities and the local exponents of the pulled-back Fuchsian equation
are straightforwardly determined from the pull-back (\ref{eq:pback}) 
and Riemann's $P$-symbol
\[
P\left\{\begin{array}{ccc} 0 & 1 & \infty
\\ 0 & 0 & a \\ 1-c & c-a-b & b
\end{array} \;z\; \right\}
\]  
of hypergeometic equation (\ref{eq:hpge}). 
For the pulled-back Fuchsian equation to have only $n$ singularities,
we usually need the local exponent differences $c-a-b$, $1-c$, $b-a$
to be inverse integers $\pm1/k,\pm1/\ell,\pm1/m$,
and the covering $z=\varphi(x)$ to be a $(k,\ell,m)$-minus-$n$ regular Belyi map.
The canonical Fuchsian equations with $n\le 4$ are hypergeometric and Heun equations.
% Singularities and local exponents of the pulled-back Fuchsian equation
% can be determined from the pullback $z\mapsto \varphi(x),\ldots$. 

We extend Definition \ref{df:klmnregular} to AB-maps.
\begin{definition} \label{df:abklmn} \rm
Given positive integers $k,\ell,m,n$, an AB-map $\varphi:\PP_x^1\to\PP_z^1$
is called {\em $(k,\ell,m)$-minus-$n$ regular} if, 
with exactly $n$ exceptions in total, 
all points above $z=1$ have branching order $k$, 
all points above $z=0$ have branching order $\ell$, and
all points above $z=\infty$ have branching order $m$. 
\end{definition} 

\begin{example} \rm
The AB-map $\varphi_1(x)$ in Example \ref{eq:d6example} 
is $(3,2,5)$-minus-4 regular. The 4 exceptional points are $x=\infty$ and 
the 3 simple roots of $\varphi_1(x)-1$.
\end{example}

\begin{remark} \rm \label{rm:vhoeij}
Recently, van Hoeij and Kunwar classified $(2,3,\infty)$-minus-5 regular AB-maps 
in \cite{vHK5}. Here $\infty$ means that all points in the third fiber are counted 
as exceptional (towards 5). These maps have degree $\le 12$.
A portion of the AB-maps $N_1,\ldots,N_{68}$ in \cite[Table 1]{vHK5} 
are applicable as $(2,3,m)$-minus-4 maps to the Fuchsian equations considered here;
see the fifth column in Table \ref{tb:abmaps}. % \S \ref{sec:algp6}.
\end{remark}

Pull-back transformations with respect to $(k,\ell,m)$-minus-$n$ regular
AB-maps can transform hypergeometric equation (\ref{eq:hpge})
with the local exponent differences $1/k,1/\ell,1/m$ to Fuchsian equations 
with an apparent singularity and $n$ other singularities.
The apparent singular point will have the local exponents $0, 2$,
rather than $0,1$ for regular points. Since AB-maps are parametrized
by algebraic curves, a generic pull-back transformation will give
isomonodromic families of Fuchsian equations with these singularities.

An important case is Fuchsian ordinary differential equations 
with an apparent singularity and $n=4$ other singular points.
Isomonodromic families of these equations are parametrized
by solutions of the Painlev\'e VI equation
\begin{align}  \label{eq:pvi}
\frac{d^2q}{dt^2}=\,& {1\over2} \! \left({1\over q}+{1\over q-1}+{1\over q-t}\right) \!
\left(\frac{dq}{dt}\right)^{\!2} \!
-\left({1\over t}+{1\over t-1}+{1\over q-t}\right)\frac{dq}{dt} \nonumber \\
&+{q(q-1)(q-t)\over t^2(t-1)^2}\left(\alpha+\beta {t\over q^2}+
\gamma{t-1\over(q-1)^2}+\delta {t(t-1)\over(q-t)^2}\right).
\end{align} 
% It has the Painlev\'e property: the only movable singularities are poles.
By the Jimbo-Miwa correspondence \cite{JM81}, a solution $q(t)$ parametrizes 
isomonodromic $2\times 2$ Fuchsian systems $dY/dx=A(x,t)Y$
with the singularities  $x=0$, $x=1$, $x=t$, $x=\infty$ 
and the local monodromy differences
$\theta_0,\theta_1,\theta_t,\theta_\infty$ such that 
\begin{equation} \label{eq:pvipara}
\alpha=\frac{(\theta_\infty-1)^2}2,\quad
\beta=-\frac{{\theta}_0^2}2,\quad\gamma=\frac{{\theta}_1^2}2,
\quad\delta=\frac{1-{\theta}_t^2}2.
\end{equation}
An equivalent isomonodromic family of ODEs has 1 apparent 
and 4 other singularities. To write down the parametric Fuchsian ODE explicitly,
one can use a specification of the Painlev\'e VI equation
in terms of the Hamiltonian system
\begin{equation} \label{eq:hamilt}
\frac{dq}{dt}=\frac{\partial \cH_0}{\partial p}, \qquad
\frac{dp}{dt}=-\frac{\partial \cH_0}{\partial q}.
\end{equation}
with
\begin{align}
\cH_0=& \, \frac{q\,(q-1)(q-t)}{t(t-1)} \! \left( p^2-\left(\frac{\theta_0}{q}
+\frac{\theta_1}{q-1}+\frac{\theta_t-1}{q-t}\right) \! p+\frac{\Theta}{q(q-1)}\right),
\nonumber \\
\Theta= &\, \frac{(\theta_0+\theta_1+\theta_t-\theta_\infty)(\theta_0+\theta_1+\theta_t+\theta_\infty-2)}{4}.
\end{align} 
The Painlev\'e VI equation is obtained by elimininating $p$.
The corresponding Fuchsian ODE is
\begin{align} \label{eq:ode4p1}
 \frac{d^2Y(x)}{dx^2} \,+ &
\left(\frac{1-\theta_0}{x}+\frac{1-\theta_1}{x-1}+\frac{1-\theta_t}{x-t}-\frac{1}{x-q}\right) 
\frac{dY(x)}{dx} + W_1\,Y(x)
  =0 %\\[-5pt] \nonumber
\end{align}
with
\begin{align*} % \label{eq:ode4p1}
W_1= \frac{\Theta}{x(x-1)} +
 \frac{q\,(q-1)\,p}{x(x-1)(x-q)}-\frac{t\,(t-1)\cH_0}{x(x-1)(x-t)};
\end{align*}
see \cite[pg.~169--173]{IKSY}  with $n=1$.
\begin{notation} \rm
Let $P_{VI}(\theta_0,\theta_1,\theta_t,\theta_\infty)$ 
denote the Painlev\'e VI equation (\ref{eq:pvipara}) with the parameters (\ref{eq:pvipara}).
Similarly, let $E(1-c,c-a-b,b-a)$ denote the hypergeometric equation (\ref{eq:hpge})
by the local exponent differences.
\end{notation}

\subsection{Computational methods}
\label{sec:compute}

% Useful differential relations for Belyi maps occur when considering
% the problem of computation of Belyi maps.

As considered in \cite[\S 5.2]{vHVHeun},
a $(k,\ell,m)$-minus-$n$ regular Belyi map has the forms
\begin{align} \label{eq:phi1phi}
\varphi(x)= &\, r_1\,\frac{P^\ell \, F}{Q^m G} \\
\label{eq:phi2phi}
= & \, 1+r_2\,\frac{R^k \, H}{Q^m G},
\end{align}
where $P,Q,R$ are monic polynomials without multiple roots;
$F,G,H$ are monic polynomials with $n$ or $n-1$ distinct roots
in total; and $r_1,r_2$ are constants.
We refer to the polynomials $P,Q,R,F,G,H$ as {\em polynomial components}
of $\varphi$.

The total number of distinct roots of 
the 6 polynomial components %$P,Q,R,U,V,W$
(including $x=\infty$ if one of the 3 terms in the polynomial
identity is of lower degree)  equals $\deg\varphi+2$, by \S\ref{sec:nomenc}\refpart{ii}.
The two expressions (\ref{eq:phi1phi})--(\ref{eq:phi2phi}) 
are equivalent to the polynomial identity 
\begin{equation} \label{eq:abc}
r_1P^\ell\,F=Q^{m}\,G+r_2R^k\,H.
\end{equation}
A $(k,\ell,m)$-minus-$n$ regular {\em AB-map} has the same shape,
but the total number of roots in the terms (including $x=\infty$) 
equals $\deg\varphi+3$ rather than $\deg\varphi+2$.
The polynomial components and the constants $r_1,r_2$
may then depend on a continuos parameter.

Polynomial identity (\ref{eq:abc}) gives % an algebraic 
a system of necessary polynomial equations   
for the undetermined coefficients of $P,Q,R$  and perhaps of $F,G,H$. 
If the degree of the target Belyi map significantly exceeds 10,
the algebraic system is too complicated, with too many degenerate 
({\em  parasytic}) solutions to be solved by Gr\"obner basis techniques efficiently.
Simpler algebraic systems are obtained by considering the logarithmic derivatives
\begin{align} \label{eq:logdif1}
\frac{\varphi'(x)}{\varphi(x)} % =h_1 \frac{R^{k-1}\,W}{P\,Q\,S}
= &\, \ell\,\frac{P'}{P}+\frac{F'}{F}-m\,\frac{Q'}{Q}-\frac{G'}{G}, \\
\frac{\varphi'(x)}{\varphi(x)-1}  % =h_2 \frac{P^{\ell-1}\,U}{Q\,R\,S}
= & \, k\,\frac{R'}{R}+\frac{H'}{H}-m\,\frac{Q'}{Q}-\frac{G'}{G}.
\end{align}
The roots of $\varphi'(x)/\varphi(x)$ are the branching points outside
the fibers $\varphi(x)\in\{0,\infty\}$, with the multiplicity reduced by 1.
This consideration gives the alternative expressions
\begin{equation}  \label{eq:logdif9}
\frac{\varphi'(x)}{\varphi(x)} =h_1 \frac{R^{k-1}\,H}{P\,Q\,S}, \qquad
\frac{\varphi'(x)}{\varphi(x)-1} =h_2 \frac{P^{\ell-1}\,F}{Q\,R\,S}.
\end{equation} 
If $\varphi(x)$ is supposed to be a Belyi map, $S$ here 
equals the product of irreducible monic factors of $F\,G\,H$, each to the power 1.
If $\varphi(x)$ is an AB-map, $S$ equals this product divided by $x-q$, 
where $q$ is the (undetermined) extra branching point.
If $x=\infty$ is in the $\varphi=\infty$ fiber, then $h_1,h_2$ 
are equal to the branching order at $x=\infty$; otherwise they are (undetermend) constants.
The obtained algebraic system for the coefficients is typically over-determined,
with fewer degenerate solutions. 
According to \cite{Couv99}, \cite{sv}, these differential relations for Belyi maps
were noticed by Fricke, Atkin, Swinnerton-Dyer .

Additional algebraic equations are obtained by considering implied pull-back
transformations of second order Fuchsian equations.
In particular \cite[Lemma 5.1]{vHVHeun}, the pull-back
\begin{equation} \label{eq:qpback}
z\mapsto \varphi(x),  \qquad
y(z)\mapsto Y(x)=\left(Q^{m} G\right)^{a} \, y(\varphi(x))
\end{equation} 
transforms the hypergeometric equation (\ref{eq:hpge}) with
\[ 
a=\frac12\left(1-\frac1k-\frac1{\ell}-\frac1m\right), \quad
b=\frac12\left(1-\frac1k-\frac1{\ell}+\frac1m\right), \quad
c=1-\frac{1}{\ell}
\] 
to the Fuchsian equation
\begin{align}   \label{eq:qpback9}
& \frac{d^2Y(x)}{dx^2}+
\left(\frac{S'}{S}-\frac{F'}{\ell\,F}-\frac{G'}{m\,G}-\frac{H'}{k\,H}\right) \frac{dY(x)}{dx}
+W_2Y(x)=0
\end{align}
with 
\begin{align}  % \label{eq:qpback9}
W_2 = & \, a \left[ b \left( \frac{ h_1h_2\,P^{\ell-2}R^{k-2}\,F\,H}{Q^2 S^2}
- \frac{m^2Q'{}^2}{Q^2} -\frac{G'{}^2}{G^2} \right)+\frac{mQ''}{Q}+\frac{G''}{G}
+\qquad \right. \nonumber \\
& \left.\quad +\left(\frac1k+\frac1\ell\right)\!\frac{mQ'G'}{Q\,G}
+\left(\frac{m Q'}{Q}+\frac{G'}{G}\right) \!
\left(\frac{S'}{S}-\frac{F'}{\ell\,F}-\frac{G'}{G}-\frac{H'}{k\,H}\right)
\right]. \nonumber
\end{align}
In the context of pull-back transformations to isomonodromic Fuchsian system
with one apparent singularity and 4 other singularities, this equation can be compared
with (\ref{eq:ode4p1}). 

\subsection{Relation to algebraic Painlev\'e VI solutions}
\label{sec:relap6}

% Kitaev \cite{Kit1}, \cite{Kit2} originated systematic study of 
% almost Belyi maps for investigating algebraic Painlev\'e VI
% solutions and isomonodromic Fuchsian equations.

Kitaev \cite{Kit1}, \cite{Kit2} initiated study of AB-maps
with the purpose of constructing algebraic Painlev\'e VI solutions.
The relevant AB maps are $(k,\ell,m)$-minus-4 
regular, % AB-maps % almost Belyi maps, 
as they induce pull-back transformations of 
to isomonodromic $2\times 2$ Fuchsian systems 
with 4 singularities (or the corresponding ODEs) % \todo{This section is to be adjusted.}
by the Jimbo-Miwa correspondence \cite{JM81} 
of these systems to Painlev\'e VI solutions.
Kitaev's basic construction  gives the following result.
\begin{theorem} \label{th:kitaev}
Let $\varphi(X)$ denote a $(k,\ell,m)$-minus-$4$ regular AB-map. % almost Belyi map.
Suppose that its irregular branching points are $X=0$, $X=1$, $X=\infty$, $X=t$.
Let $X=q$ denote the extra branching point of order $2$. 
Then $q(t)$ is an algebraic Painlev\'e VI solution with the parameters 
%$(\alpha,\beta,\gamma,\delta)$ 
$\theta_j=a_j/K_j$ for $j\in\{0,1,t\}$, and $\theta_\infty=1-a_\infty/K_\infty$.
Here $K_j,K_\infty\in\{k,\ell,m\}$ depending on the fiber of each 
of the $4$ singularities, and $a_j,a_\infty$ are the branching orders 
at them. % $4$ exceptional points.
\end{theorem} 
\begin{proof}
% The proof is by 
The Jimbo-Miwa correspondence \cite{JM81} and 
explicit consideration of a pull-back from $E(1/\ell,1/k,1/m)$.
This is the particular case $\varepsilon=1$ of \cite[Theorem 2.1]{Kit1}.
\end{proof}
This theorem gives differential relations between {\em coefficients} of AB-maps.
The relation between $t$ and $q$ is algebraic
because the Hurwitz space is one-dimensional.
\begin{example} \rm \label{ex:lt08}
Consider the polynomials
\begin{align}
P= &\; x^4+4\w x^2-6\w x+\w^2, \nonumber \\
R= &\; 2x^6+12\w x^4-18\w x^3+15\w^2x^2-36\w^2x-\w^2(2\w-27), \nonumber \\
G_1 = &\; x-1, \\ 
G_2 = &\; 4x^3+\w x^2+18\w x+\w(4\w-27). \nonumber
\end{align}
Reminiscent to (\ref{eq:abc}), we have a polynomial identity
$%\begin{equation}
4 P^3 = R^2+r_0 G_1^2\,G_2
$ %\end{equation}
with \mbox{$r_0=27\w^3$.}
It defines a $(2,3,7)$-minus-4 regular AB-map
\begin{equation}
\varphi_2(x)=\frac{4P^3}{r_0 G_1^2\,G_2}
=1+\frac{R^2}{r_0 G_1^2\,G_2}
\end{equation}
of degree 12,
with the branching pattern $[2^6/3^4/7\,2\,1^3]$.
% \todo{This notation has to be defined.}
The extra branching point is $x=\q_2=(9-2\w)/7$. 
To obtain an algebraic Paineve VI solution of $P_{VI}(1/7,1/7,2/7,6/7)$ 
by Theorem \ref{th:kitaev}, we first reparametrize 
\begin{equation} \label{eq:lt08subs}
\w\mapsto -\frac{(\s^2+3)^3}{(\s-1)^2(\s+1)^2}
%-\frac{4(t^2-t+1)^3}{t^2(t-1)^2}
\end{equation} 
so that $G_2$ has rational roots:
\begin{align*}
x_1\!=\frac{(\s^2\!+3)(\s^2\!+15)}{4\,(\s-1)\,(\s+1)}, \ 
x_2\!=\frac{(\s^2\!+3)(2\s^2\!+3\s+3)}{(1-\s)\,(\s+1)^2}, \ 
x_3\!= \frac{(\s^2\!+3)(2\s^2\!-3\s+3)}{(\s-1)^2\,(\s+1)}.
%x_1= &\; \frac{(t^2-t+1)(t^2-t+4)}{t\,(t-1)}, &  x_2= -\frac{(t^2-t+1)(4t^2-t+1)}{t^2(t-1)},\\
%\\ x_3= &\; \frac{(t^2-t+1)(4t^2-7t+4)}{t\,(t-1)^2}.
\end{align*}
We move these points to the locations $X_1=\infty$, $X_2=0$, $X_3=1$
by the M\"obius $x$-transformation
\begin{equation}
x\mapsto \left(\frac{\s^2+3}{\s^2-1} \right) \!
\frac{4\s^3(\s^2+15)X-(\s-3)^3(2\s^2+3\s+3)}{16\s^3X+(\s+1)(\s-3)^3}.
\end{equation}
The root of $G_1$ is transformed to $X=t_2$ with
\begin{equation}
t_2=\frac{(\s-3)^3(\s^2+\s+2)^2}{2\s^3(\s^2+7)^2},
\end{equation}
and the transformed location 
of the extra branching point is
\begin{equation}
\q_2=\frac{(\s+1)(3-\s)(\s^2+\s+2)}{2\s(\s^2+7)}.
\end{equation}
This parametrizes an algebraic solution $\q_2(t_2)$ of  $P_{VI}(1/7,1/7,2/7,6/7)$.
The fractional-linear transformation $t_2\,(\q_2-1)/(\q_2-t_2)$ 
permutes the singularities $0\leftrightarrow1$, $t\leftrightarrow\infty$,
and gives the {\em Kleinian solution} of Boalch \cite{Boa1}.
Kitaev derived this solution by the pull-back construction \cite[\S 3.4.3]{Kit1},
also followed in \cite[\S 5]{VK2}.
\end{example}

\begin{example} \rm \label{ex:lt33}
Consider the polynomials
\begin{align}
P= &\; x^3+(\w-6)x^2+24x-48, \nonumber \\
R= &\; x^5+2(\w-6)x^4+(\w^2-12\w+72)x^3+36(\w-8)x^2-72(\w-9)x-864, \quad \nonumber \\
F = &\; x+\w-6, \\ 
G = &\; \w x^3+(\w^2-6\w-3)x^2+8(3\w+1)x-16(4\w+3). \nonumber
\end{align}
We have a polynomial identity
$%\begin{equation}
P^3\,F = R^2+1728G
$. %\end{equation}
It defines a $(2,3,7)$-minus-4 regular AB-map $\varphi_3(x)$
%\begin{equation}
%\varphi_1(x)=\frac{4P^3}{\kappa_0 V_1^2\,V_2}
%=1+\frac{R^2}{\kappa_0 V_1^2\,V_2}
%\end{equation}
of degree 10, with the branching pattern $[2^5/3^3\,1/7\,1^3]$.
The extra branching point is $x=\q_3=-4(\w^2-6\w-6)/(7\w)$. 
The curve $G(x,\w)=0$ defines a genus 0 curve;
a parametrization of it gives a substitution 
after which the polynomial $G(x)$ has a rational root:
\begin{equation}
\w\mapsto \frac{(s+2)(s^2+2s+9)}{(s-1)^2}.
%\qquad y^2=s\,(s^2+s+7).
\end{equation} 
Complete factorization of $G$ is achieved 
on the genus 1 curve $y^2=s\,(s^2+s+7)$.
Here are the roots of $G$:
\begin{align*}
x_1\!=\frac{(1-s)(s+3)}{s+2}, \quad 
x_2\!=\frac{4(2s^2+2s+5+3y)}{(s-1)\,(3-y)}, \quad
x_3\!=\frac{4(2s^2+2s+5-3y)}{(s-1)\,(3+y)}.
\end{align*}
% To obtain an algebraic Paineve VI solution of $P_{VI}(1/7,1/7,1/3,6/7)$ 
% by Theorem \ref{th:kitaev}, we reparametrize 
The three roots are mapped to $X_1=\infty$, $X_2=0$, $X_3=1$
by the M\"obius $x$-transformation
\begin{equation*}
x\mapsto 
\frac{4(1-s)\,\big( 2y(s+3)(s^2\!+s+7)(2X\!-\!1)
-3s^4\!-34s^3\!-114s^2\!-252s-245\big)}
{8y(s+2)(s^2+s+7)(2X\!-\!1)-s^6\!-2s^5\!+9s^4\!+64s^3\!+221s^2\!+210s+147}.
\end{equation*}
The root of $F$ is transformed to $X=t_3$ with
\begin{equation}
t_3=\frac12+\frac{s^9-84s^6-378s^5-1512s^4-5208s^3-7236s^2-8127s-784}
{432\,(s+1)^2\,(s^2+s+7)\,y},
\end{equation}
and the transformed location of the extra branching point is
\begin{equation}
\q_3=\frac12-\frac{s\,(s^4+2s^3+12s^2+20s+73)}{12\,(s+1)\,(s+2)\,y}.
\end{equation}
This parametrizes an algebraic solution $\q_3(t_3)$ 
of  $P_{VI}(1/7,1/7,1/3,6/7)$, of genus 1.
An equivalent solution $t_3\,(\q_3-1)/(\q_3-t_3)$ of $P_{VI}(1/7,1/7,1/7,2/3)$
was first found by Kitaev \cite[\S 3]{Kit2} by the pull-back method.
% also followed in \cite[\S 5]{VK2}.
\end{example}

More generally, Kitaev's method \cite{Kit1}
allows further {\em Schlessinger gauge transformations}
to obtain multiple algebraic Painlev\'e VI solutions from the same pull-back transformation.
These transformations are matrix analogues of (\ref{eq:pback}) with $\varphi(x)=x$.
They shift local exponent differences 
(including $\theta_0,\theta_1,\theta_t,\theta_\infty$) by integers;
the total shift sum must be even. The whole construction is called {\em RS-transformations},
where R stands for a Rational pull-back, and S stands for a Schlessinger %gauge 
transformation.

\begin{example} \rm \label{ex:kitaevshift}
Examples \ref{ex:lt08}, \ref{ex:lt33} implicitly employ pull-backs of 
the hypergeometric equation $E(1/2,1/3,1/7)$
to isomonodromic Fuchsian equations with 4 
singularities at the roots of $V_1,V_2$ (or $U,V$, respectively) 
and an apparent singularity at $x=\q_2$ (or $x=\q_3$).
This lead to algebraic solutions of $P_{VI}(1/7,1/7,2/7,6/7)$ 
and $P_{VI}(1/7,1/7,1/3,6/7)$ by Theorem \ref{th:kitaev}.
%by Kitaev's construction. 
The same pull-back transformations can be applied to 
the hypergeometric equations $E(1/2,1/3,2/7)$ and $E(1/2,1/3,3/7)$,
as suggested by Kitaev \cite{Kit1}, \cite{Kit2}.
The pull-backs of $E(1/2,1/3,2/7)$ have the same $4+1$ singularities,
plus a new apparent singularity at $x=\infty$.
Schle\-ssinger transformations neutralizing this singularity
give algebraic solutions of $P_{VI}(2/7,2/7,4/7,2/7)$, $P_{VI}(2/7,2/7,1/3,2/7)$,
as demonstrated in \cite{VK2}. 
Similarly, the pull-backs of $E(1/2,1/3,3/7)$ have the same $4+1$ singularities,
plus a new singularity at $x=\infty$ with the monodromy difference 3.
Neutralizing Schle\-ssinger transformations lead to algebraic solutions
of $P_{VI}(3/7,3/7,6/7,4/7)$ and $P_{VI}(3/7,3/7,1/3,4/7)$,
as shown in \cite{VK2}.
\end{example}
It is worth recalling here the {\em Okamoto} (also called {\em B\"acklund}) 
{\em transformations} \cite{Oka}
that convert $q(t)$ to rational functions  of $q(t)$, $dq/dt$ and $t$. 
The basic transformation acts on the parameters of the Painlev\'e VI equation as follows:
\begin{align} \label{eq:bokamoto}
(\theta_0,\theta_1,\theta_t,\theta_\infty) \mapsto 
(\Theta-\theta_0,\Theta-\theta_1,\Theta-\theta_t,\Theta-\theta_\infty),
\end{align}
with $\Theta=(\theta_0+\theta_1+\theta_t+\theta_\infty)/2$.
Special cases are transformations that shift $(\theta_0,\theta_1,\theta_t,\theta_\infty)$ 
by integer vectors, with the total shift even. 
They can be realized by Schlessinger gauge transformations
of the Fuchsian equations.

Note that $P_{VI}(\pm\theta_0,\pm\theta_1,\pm\theta_t,1\pm\vartheta_\infty)$
is the same Painlev\'e VI equation, hence (\ref{eq:bokamoto}) defines 16 ``neighbouring"
Painlev\'e VI equations by Okamoto %-Backl\"und 
transformations.
A set of fractional linear transformations permutes the 4 singular points.
All together \cite{Oka}, these transformations form an affine Weyl group of type $E_6$.
Up to the integer shifts and permutation of the singular points, 
a generic Okamoto orbit contains three distinct Painlev\'e VI solutions.
\begin{example} \rm \label{eq:backlund}
The equations 
\begin{equation*}
P_{VI}(1/7,1/7,2/7,6/7), \quad P_{VI}(2/7,2/7,4/7,2/7), 
\quad P_{VI}(3/7,3/7,6/7,4/7)
\end{equation*} 
in Example \ref{ex:kitaevshift} 
and their algebraic solutions are related  by the Okamoto %-Backl\"und 
transformations. % (\ref{eq:bokamoto}). 
But the equations 
\begin{equation*}
P_{VI}(1/7,1/7,1/3,6/7), \quad
P_{VI}(2/7,2/7,1/3,2/7), \quad P_{VI}(3/7,3/7,1/3,4/7)
\end{equation*} 
are not related 
by the Okamoto %-Bakcl\"und 
transformations. % In particular, %Up to permutation of the 4 singular points, 
For example, the Okamoto %-Backl\"und 
orbit of $P_{VI}(1/7,1/7,1/3,6/7)$  consists 
of Schlessinger and fractional-linear transformations  of itself %$P_{VI}(1/7,1/7,1/3,6/7)$, 
and $P_{VI}(17/42,17/42,17/42,5/42)$, $P_{VI}(11/42,11/42,11/42,23/42)$.
% Kitaev's construction does not cover all orbits...
\end{example}

\section{Differentiation relations from free divisors}
\label{sec:freediv}

Theorem \ref{th:kitaev} gives differential relations between coefficients of  AB-maps.
Here we observe differential relations 
with differentiations both with respect to the variable $x$ and a parameter $\w$.

\subsection{Free divisors, logarithmic vector fields}
% \todo{Title change?}

As presented in \cite{kms15}, interesting examples of flat structures, free divisors in 
the sense of Saito  \cite{Saito}  % // Saito's singularity theory \cite{Saito} 
can be constructed from algebraic Painlev\'e VI solutions. %  and corresponding AB-maps. 
In Dubrovin's context \cite{Dubr} of 
Frobenius manifolds,  the potentials which are 
solutions of the Witten-Dijkgraaf-Verlinde-Verlinde equations play a similar key role.

As discussed in \S \ref{sec:relap6}, the use of AB-maps is one of the methods 
to construct algebraic Painlev\'e VI solutions.
For these reasons, it is meaningful to study a relationship between AB-maps 
and free divisors. 
As an observation by comparing AB-maps with free divisors, 
we recognized that after a suitable homogenization of variables of $(k,\ell,m)$-minus-4
regular AB-maps, polynomials which define free divisors appear as 
polynomial components of AB-maps. % functions. 
We explain this observation by taking the following example.
% The polynomial components of $(k,\ell,m)$-minus-4 regular AB-maps
% that give the 4 exceptional points typically define {\em free divisors}
% after a suitable homogenization. 
\begin{example} \rm \label{ex:lt18}
We homogenize the AB-map $\varphi_1$ of Example \ref{eq:d6example}
by  $w=v/u^3$, \mbox{$x=uX/v^2$} % $(\w,x)\mapsto (v/u^3 ,  uX/v^2)$ 
with the variables $u,v,X$ of  weights % the weighted degrees 
$1,3,5$, respectively. The weighted-homogeneous polynomials are
\begin{align*}
P= & \,X, \qquad % \hspace{140pt} 
Q = u\,X+v^2, \\
R= & \, X^3+15u^2vX^2+20uv^3X+8v^5,\\
F= & \,  X^3+2u^2(15v-32u^3)X^2+5uv^2(8v-19u^3)X+8(2v-5u^3)v^4.
\end{align*}
Correspondingly, they satisfy $P^3F+64Q^5=R^2$. 
Let us consider the vector fields % differential operators
\begin{align} \label{eq:euler}
\V_1= &\, u\,\frac{\partial}{\partial u}+3v\,\frac{\partial}{\partial v}
+5X\,\frac{\partial}{\partial X},  \\  %  \label{eq:wonder}
\V_2= &-2 (v-3u^3) \frac{\partial}{\partial u} + (X+3u^2v) \frac{\partial}{\partial v},
\label{eq:wonder2} \\ 
\V_3= &\, 3(X+27u^2v-64u^5) \frac{\partial}{\partial u} 
+8u(7v-12u^3)v \frac{\partial}{\partial v}-40v^3  \frac{\partial}{\partial X}. 
\end{align}
They are {\em logarithmic} along the hypersurface $F=0$,
meaning that their action on % the weighted-homogeneous % polynomial 
the polynomial $F$ coincides with some polynomial multiplication:
\begin{align}  \label{eq:logarithm}
\V_1\,F = 15F, \qquad
\V_2\,F = 30u^2F, \qquad
\V_3\,F = 60(3v-16u^3)F. % \nonumber
\end{align}
Consider the matrix 
\begin{equation}
M = \left( \begin{array}{ccc}
u & 3v & 5X \\
-2 (v-3u^3) & X+3u^2v & 0 \\
3(X+27u^2v-64u^5) & 8u(7v-12u^3)v & -40v^3
\end{array} \right)
\end{equation}
where the rows represent the vector fields, so that
\[
\left( \begin{array}{c} \V_1 \\ \V_2 \\ \V_3 \end{array} \right)= 
 M \left(  \begin{array}{c} \partial/\partial u \\[1pt] \partial/\partial v \\[1pt] 
 \partial/\partial X \end{array}  \right).
\]
Then $\det M = -15F$.
Existence of 3 logarithmic vector fields along % the hypersurface 
$F=0$, % with the logarithmic action (\ref{eq:logarithm})
and the identification of $F$ with $\det M$ % their determinant 
up to a constant multiple means that the hypersurface $F=0$ is a {\em free divisor}.
More conceptually \cite{MoSc}, a characteristic property is %means 
that the logarithmic vector fields form a free module over $\CC[u,v,X]$.

The Euler vector field $\V_1$ acts on the other polynomial components $P,Q,R$
as multiplication by the weighted-homogeneous degrees $5,6,9$ (respectively).
Remarkably, the vector field $\V_2$ is logarithmic 
% acts on the remaining polynomials $P,Q,R$ logarithmically 
along the hypersurfaces $P=0$, $Q=0$, $R=0$ as well:
\begin{equation}  \label{eq:logvect}
\V_2\,P = 0, \qquad
\V_2\,Q = 6u^2Q, \qquad
\V_2\,R = 15u^2R. % \nonumber
\end{equation}
This special role of $V_2$ is unexpected.

The isomonodromic Fuchsian system can be elegantly expressed in terms
of the vector fields
\begin{equation}
\widetilde{\V}_2=V_2-2u^2\,V_1, \qquad \widetilde{\V}_3=V_3+32u^2\,V_2-12uv\,V_1.
\end{equation}
The action on the AB-map
\begin{equation}
\widetilde{\varphi}_1=-\frac{P^3F}{64\,Q^5}
\end{equation}
is
\begin{equation}
V_1\,\widetilde{\varphi}_1=0, \qquad
\widetilde{\V}_2\,\widetilde{\varphi}_1=0, \qquad
\widetilde{\V}_3\,\widetilde{\varphi}_1=-\frac{15\,R}{P\,Q}\,\widetilde{\varphi}_1,
\end{equation}
and the pulled-back hypergeometric function
\begin{equation}
f=Q^{1/12}\,F^{\lambda/15}\, % Q^{\frac{1}{12}}\,U^{\frac{r}{15}-\frac1{30}}\,
\hpg21{-1/60,\,11/60}{2/3}  % {-\frac1{60},\frac{11}{60}}{\frac23}
{\widetilde{\varphi}_1} % \quad\mbox{with}\quad
\end{equation}
satisfies the differential system
\begin{align} \label{eq:vvv}
V_{1\,}f= &\; \big(\lambda+\textstyle\frac12\big) f, %\frac{r}{2}\,f,
\nonumber \\ \widetilde{V}_{2\,}f= & \, -{\textstyle \frac{1}2} \,u^2f, \\
\widetilde{V}_3^{\,2}f= & \, -\big( (9v+20u^3)X+30u^2v^2 \big) f. \nonumber
\end{align}
% The vector fields $V_1,\widetilde{V}_2$ annihilate the AB-map $\widetilde{\varphi}_1$.
The last equation has order 2, 
just as % and presents the pull-back equivalence to 
the hypergeometic equation.
\end{example}
As free divisors and AB-maps are defined for many algebraic Painlev\'e VI solutions,
we checked that attractive differential systems like (\ref{eq:vvv}) for pulled-back
hypergeometric solutions exists in every computed (and homogenized) case.
The computed cases are presented in \S \ref{sec:algp6}.
Existence of ``universally" logarithmic vector fields 
as in (\ref{eq:logvect}) was observed as well.
\begin{observe} \label{rm:observe}
For every computed AB-map $\varphi(X\!:\!u\!:\!v)$
in weighted-homogeneous variables $u,v,X$ of the 
respective weights $N_X,N_u,N_v$, there is a vector field 
\begin{equation} \label{eq:vf}
\widetilde{A}(X,u,v)\frac{\partial}{\partial X}+\widetilde{B}(X,u,v)\frac{\partial}{\partial u}
+\widetilde{C}(X,u,v)\frac{\partial}{\partial v}
\end{equation}
linearly independent from the Euler vector field 
\begin{equation} \label{eq:euler2}
N_x\,x\,\frac{\partial}{\partial x}+
N_u\,u\,\frac{\partial}{\partial u}+N_v\,v\,\frac{\partial}{\partial v}
\end{equation}
that acts by polynomial multiplication 
on {\em all} polynomial components of $\varphi$. 
%(x:u:v)$ as in $(\ref{eq:logvect})$. 
% logarithmically (that is, ).
\end{observe}
This observation is remarkable. As a weaker implication, 
it says that there are low degree syzygies between 
$\partial U/\partial X$, $\partial U/\partial u$, $\partial U/\partial v$ 
and $U$ for any polynomial component $U\in\{P,Q,R,F,G,H\}$ 
as in (\ref{eq:phi1phi})--(\ref{eq:phi2phi}).
We found the exceptional vector fields by computing the lowest degree syzygy 
between the derivatives of $R$, and checking the observation
on other polynomial components. The chosen syzygy is always much smaller
than alternatives.

\subsection{Dehomogenization}

Observation \ref{rm:observe} can be modified to apply to non-homogeneous 
AB-maps $\varphi(x,\w)$. The modified claim is that there exists a single vector field 
that acts by polynomial multiplication on all polynomial components of $\varphi$.
If de-homogenization of $\varphi(X\!:\!u\!:\!v)$ is simply $u=1$,
one can find the ``universally" logarithmic vector field as 
a linear combination of (\ref{eq:vf}) and (\ref{eq:euler2})
with eliminated $\partial/\partial u$, and then specialize to $u=1$.
\begin{example} \rm 
Recall Example \ref{eq:d6example} and consider the vector field
\begin{equation}
\OP_1=-2x(x+1) \frac{\partial}{\partial x}+ (\w x+6\w-15) \frac{\partial}{\partial w}.
\end{equation}
This vector field acts acts by polynomial multiplication 
on all polynomial components of $\varphi_1(x)$,
% and $\varphi_1(x)-1$ logarithmically,  
including on $x$ and $x+1$. 
To derive this vector field from Example \ref{ex:lt18},
we substitute 
\begin{align*}
 \frac{\partial}{\partial u} 
& = \frac{x}{v^2} \frac{\partial}{\partial x}-\frac{3v}{u^4} \frac{\partial}{\partial\w}
 \qquad = \frac1u \left(x\frac{\partial}{\partial x}-3\w \frac{\partial}{\partial\w} \right), \\
\frac{\partial}{\partial v} 
& =-\frac{2uX}{v^3}\, \frac{\partial}{\partial x}+\frac{1}{u^3} \frac{\partial}{\partial\w}
 =  \frac1v \left( -2x\frac{\partial}{\partial x}+w\frac{\partial}{\partial\w} \right)
\end{align*}
into $V_2$, and recognize $\OP_1$ after multiplication by $u/v$.
\end{example}
\begin{example} \rm
For Example \ref{ex:lt08}, the vector field
\begin{equation}
\OP_2=(x-1)(3x+\w) \frac{\partial}{\partial x}+\w\,(7x+2\w-9) \frac{\partial}{\partial \w}
\end{equation}
acts on $P,R,G_1,G_2$ and even on $r_0$ % logarithmically.
by polynomial multiplication. 
Considering % the polynomial components of 
$\varphi_2(x,s)$  after the substitution (\ref{eq:lt08subs}),
the vector field
\begin{equation*}
\widetilde{\OP}_2=-14(s^2+3)x(x-1)\frac{\partial}{\partial x}
+\big( 2s(s^2+7)x+(s+1)(s-3)(s^2+s+2)\big) \frac{\partial}{\partial s}
\end{equation*}
is logarithmic for every polynomial component (with cleared denominators $\in\QQ[s]$).
% logarithmically as well. 
As $G_2$ factors $(x-x_1)(x-x_2)(x-x_3)$ over $\QQ(s)$,
the vector field % $\widetilde{\OP}_2$ 
is logarithmic even along the hypersurfaces $x-x_k=0$ for $k\in\{1,2,3\}$. 
% all three factors (with cleared denominators). % $\in\QQ[s]$ as well.

For Example \ref{ex:lt33}, the vector field
\begin{equation}
\OP_3=(x^2-2(\w+3)x+24) \frac{\partial}{\partial x}+
(7\w x+4\w^2-24\w-24) \frac{\partial}{\partial\w}
\end{equation}
acts on $P,R,F,G$ by polynomial multiplication.
\end{example}

If a vector field is logarithmic along two hypersurfaces $F=0$, $G=0$, 
it is logarithmic along $FG=0$ as well. In the observed examples,
the exceptional vector fields annihilate the AB-maps $\varphi_j$. 
Consequently, those vector fields can be normalized to 
\begin{equation} \label{eq:vfnh}
{A}(x,\w)\frac{\partial}{\partial x}+{B}(x,\w)\frac{\partial}{\partial \w},
\end{equation}
with ${A}(x,\w)=\partial\varphi_j/\partial\w$ and 
$B(x,\w)=-\partial\varphi_j/\partial x$. 
This explains why the coefficient to $\partial/\partial\w$ or $\partial/\partial s$ 
is linear in $x$  in the above examples, and the roots are the extra branching points 
$q_j$ for $j\in\{1,2,3\}$. The extra branching point is the only root
of $\partial\varphi_j/\partial x$ that is not a root of $\partial\varphi_j/\partial\w$. 

Observation \ref{rm:observe} becomes simpler in a dehomogenized form.
With more specificity, we formulate the following conjecture. 
\begin{conjecture} \label{th:conj}
For any AB-map $\varphi(x,\w)$ with a field of definition $K=\QQ(w)$, 
there exists a vector field $(\ref{eq:vfnh})$ 
% \begin{equation} \label{eq:vfnh}
% {A}(x,\w)\frac{\partial}{\partial x}+{B}(x,\w)\frac{\partial}{\partial \w}
% \end{equation}
% annihilating $\varphi(x,\w)$  
that acts on every $K[x]$-irreducible factor 
of the numerators and the denominators of  $\varphi$ and $\varphi-1$
by polynomial multiplication. The vector field annihilates $\varphi(x,\w)$.
\end{conjecture}
As exemplified above, the conjecture implies that $B(x,\w)$ is linear in $x$, 
and its root gives  the extra branching point of $\varphi$ outside 
the critical fibers $\{0,1,\infty\}$.
By the asymptotics at $x=\infty$, the degree of $A(x,\w)$ in $x$ is at most $2$.

We checked the conjecture for all known AB-maps, including the $(2,3,\infty)$-minus-5
maps from \cite{kms15} that we mentioned in Remark \ref{rm:vhoeij}.
Explicit prior knowledge of these vector fields should be very useful 
in speeding up computation of a desired AB-map,
by utilizing new algebraic equations for undetermined coefficients.

\section{Algebraic Painlev\'e VI solutions}
\label{sec:algp6}

Algebraic solutions of the Painlev\'e VI equation were recently classified by
Lisovyy and Tykhyy \cite{LiTy}.  %\todo{To be polished}
Apart from infinite families of rational or Picard's $P_{VI}(0,0,0,1)$ solutions
presented in \cite[Propositions 49, 51]{LiTy} and their Okamoto orbits,  
% and an infinite family of Picard, Hitchin's solutions \cite{Hit}
there is a finite list (up to Okamoto %-Backl\"und 
transformations) of 3 parametric and 45 non-parametric solutions. 
The non-parametric solutions were already derived by Dubrovin, Mazzocco \cite{DuMo},
Kitaev \cite{Kit1}, \cite{Kit2} and Boalch \cite{Boa1}, \cite{Boa2},  \cite{Boa3} in 2000--2007.
% The solutions have genus 0, 1, 2, 3 or  7. 

% Most compactly, an algebraic solution $y(t)$ is represented by 
% a parametrization $y(s), t(s)\in\CC(s)$ for genus 0 solutions, or
% $y=y(s,w)$, $t=t(s,w)$ on an algebraic curve $\Psi(w,s)=0$.

\subsection{AB-maps for algebraic Painlev\'e VI solutions}

Kitaev conjectured \cite{K1}
that all algebraic solutions of the Painlev\'e VI equation
can be obtained from pull-back transformations by $(k,\ell,m)$-minus-4 
regular AB-maps, up to Okamoto %-Backl\"und transformations 
and Schlessinger % gauge 
transformations.
By checking the Lisovyy-Tykhyy classification we see that this conjecture is true
for the $3+45$ solutions in  \cite{LiTy}:
\begin{enumerate}
\item The 3 Okamoto %-Backl\"und 
orbits \#II\,--\,\#IV of parametric solutions have corresponding pull-back transformations,
as first established in \cite{AK}. % by Kitaev \cite{Kit1}.
% and Example \ref{ex:parametric} here below.
\item The Lisovyy-Tykhyy solutions \#8, \#33 are obtained by 
the pull-back maps $\varphi_2(x)$, $\varphi_3(x)$ of Examples \ref{ex:lt08}, \ref{ex:lt33}.
The similar solutions \#32, \#34 solve $P_{VI}(2/7,2/7,1/3,2/7)$ and $P_{VI}(3/7,3/7,1/3,4/7)$.
They are obtained from $\varphi_3(x)$ by additional Schlessinger transformations 
described in Example \ref{ex:kitaevshift}.
%  obtained by the pull-back transformations
% $\varphi_2(x)$, $\varphi_3(x)$ of Examples \ref{ex:lt08}--\ref{eq:backlund}.
\item The other % Lisovyy-Tykhyy 
solutions in \cite{LiTy} correspond (up to Okamoto %-Backl\"und 
transformations) to isomonodromic Fuchsian equations with finite monodromy. 
Existence of pull-backs is implied by celebrated Klein's theorem \cite{klein77}:
any second order Fuchsian equations with finite monodromy is a pull-back 
of a hypergeometric equation with finite monodromy. 
% \item The rational, Picard and Hitchin solutions... \todo{have to be checked}
\end{enumerate}
In \refpart{iii}, there are  33 Okamoto %-Backl\"und 
orbits corresponding to Fuchsian systems with the icosahedral monodromy group; 
and 7 octahedral (\#4, \#5, \#9, \#10, \#20, \#21, \#30),  1 tetrahedral (\#3) cases.
As Schlessinger transformations do not change monodromy of Fuchsian equations,
the exponent differences $\theta_0,\theta_1,\theta_t,\theta_\infty$ can be shifted by integers.
This gives infinitely many Kleinian pull-backs by AB-maps of unbounded degree
in these Okamoto %-Backl\"und 
orbits. Okamoto transformtions are necessary, 
as (for example, \#16, \#17, \#31 in \cite{LiTy})
the Dubrovin-Mazzocco % \cite{DuMo} 
solutions of $P_{VI}(0,0,0,4/5)$, $P_{VI}(0,0,0,2/5)$, $P_{VI}(0,0,0,2/3)$ 
correspond to Fuchsian systems with logarithmic singularities
and cannot be obtained directly by a pull-back transformation.

\begin{table}
\renewcommand*{\baselinestretch}{1.02}
\vspace{-10pt}  \hspace{-34pt}
\begin{tabular}{@{}rllllcr@{}}
\hline
\multicolumn{3}{l|}{Painlev\'e VI solution} & % ref. & 
\multicolumn{2}{l|}{Almost Belyi map} &
\multicolumn{2}{l}{Braid monodromy} \\  \hline
\cite{LiTy}\!\!\!  & % $d$ & $g$ 
Monodromy & \multicolumn{1}{c|}{Exp. differences\!} &
Passport & \multicolumn{1}{l|}{Ref. or $d$} & Passport & $d^*$  \\ \hline
II % & 2 & 0 
& $2/\!/1^2$ % & ---
& $a,a,b,1-b$ % & \multicolumn{2}{c}{$(b+\frc12,b+\frc12,a+\frc12,\frc12-a)$} 
% \multicolumn{2}{c}{$(\frc12,\frc12,a-b+\frc12,\frc12-a-b)$} & 2 
& $1^2/1^2/2$ & $\cite{AK}, N_ 1/N_2\!\!$ & 1/1/1 & 1 \\ %1/\!/\!/  \\
III % & 3 & 0 
& $2\,1/\!/3$ % & D 
& $a,a,2a,\frc23$ % & \multicolumn{2}{c}{$(b,b,\frc13,2b+1)$, $b=a\pm\frc13$} & 4
& $1^2\,2/3\,1/2^2$ & \cite{AK},$\,N_7/N_9\!\!$ &  $3/2\,1/2\,1$ & 3 \\ % $2\,1/\!/3$ \\
IV % & 4 & 0 
& $3\,1/\!/\!/$ % & D 
& $a,a,a,\frc12$ % & & 3 
& $1^3/2\,1/3$ & \cite{AK},\,$N_3/N_4\!\!$ & 1/1/1 & 1 \\ %  1/\!/\!/  \\
& \multicolumn{1}{r}{$(a\!=\!2b\pm\frac12)\!\!$} & $b,b,b,1-3b$ % , $2b=a\pm\frc12$} & 6 
& $1^3\,3/3^2/2^3$ & \cite{Kit1},$\,N_{23}$ & $2/2/1^2$ & 2 \\ % $2/\!/1^2$ \\
1 % & 5 & 0 
& $3\,2/\!/2^21$ % & I20 
& $\frc13,\frc13,\frc15,\frc35$ % &  $(\frc1{5},\frc1{5},\frc1{15},\frc{11}{15})$ 
% & $(\frc25,\frc25,\frc{7}{15},\frc{13}{15})$ & 8 
& $3^21^2/5\,1\,2/2^4$ & $N_{33}$ & $7\,3/4\,3\,2\,1/2^41^2$ & 10 \\
2 % & 5 & 0 
& $3\,2/\!/3\,1^2$ % & K,I21 
& $\frc15,\frc15,\frc25,\frc25$ % & $(\frc2{5},\frc2{5},\frc1{5},\frc{1}{5})$ 
% & $(0,0,\frc{1}{5},\frc{3}{5})$ & 12 
& $5\,1^22\,3/3^4/2^6$ & \cite{Kit1},$\,N_{61}$ & $6\,5\,4/3^42\,1/2^71$ & 15 \\
3 % & 6 & 0 
& $3^2/\!/2^21^2$ % & AK & $(\frc12,\frc12,\frc13,\frc13)$ 
& $\frc1{3},\frc1{3},\frc1{2},\frc{1}{2}$ % & $(0,0,\frc{1}{6},\frc{5}{6})$ & 4 
& $3\,1/3\,1/2\,1^2$ & \cite{AK},$\,N_6$ & $4\,2/4\,2/3\,1^3$ & 6 \\
& & $\frc1{2},\frc1{2},\frc1{3},\frc{1}{3}$ & 
$2^21^2\!/3\,1\,2/3^2$ & $N_{19}$ & $5\,3\,2^2/3^32\,1/3^32\,1$ & 12 \\
4 % & 6 & 0 
& $4\,2/3^2/2^21^2$ % & B7 & $(\frc12,\frc12,\frc13,\frc14)$ & 
& $\frc14,\frc12,\frc13,\frc12$ % $(\frc1{24},\frc1{24},\frc5{24},\frc{19}{24})$ 
% & $(\frc{7}{24},\frc{7}{24},\frc{11}{24},\frc{13}{24})$ & 7 
& $4\,1\,2/3^21/2^{3}1$ & $N_{28}$ &  $6^25\,3^21/4^23^42^2/3^32^61^3$ & 24 \\
5 % & 6 & 0 
& $4\,2/\!/3\,1^3$ % & K & $(\frc13,\frc13,\frc14,\frc14)$ 
& $\frc1{4},\frc1{4},\frc1{3},\frc{1}{3}$ % & $(0,0,\frc{1}{12},\frc{7}{12})$ & 6 
& $4\,1^2/3\,1\,2/2^{3}$ & \cite{Kit1},$\,N_{24}$ & $5\,3\,1/5\,3\,1/2^41$ & 9 \\
 & & $\frc1{3},\frc1{3},\frc1{4},\frc{1}{4}$ % & $(0,0,\frc{1}{12},\frc{7}{12})$ & 6 
& $3^21^2/4\,1\,3/2^{4}$ & $N_{34}$ & $7\,4\,3\,1/4^23^21/2^71$ & 15 \\
6 % & 6 & 0 
& $3\,2\,1/\!/5\,1$ % & I23 
& $\frc15,\frc25,\frc25,\frc23$ % & $(\frc1{6},\frc7{30},\frc7{30},\frc{29}{30})$ 
% & $(\frc{1}{6},\frc{13}{30},\frc{13}{30},\frc{19}{30})$ & 10 
& $5\,1\,2^2/3^31/2^{5}$ & $N_{52}$ & $7\,3^22/4\,3^32/2^71$ & 15 \\
7 & idem % I22 & 
& $\frc13,\frc15,\frc15,\frc25$ % $(\frc25,\frc15,\frc15,\frc13)$ &
% $(\frc1{6},\frc1{30},\frc1{30},\frc{17}{30})$ 
% & $(\frc{1}{6},\frc{11}{30},\frc{11}{30},\frc{7}{30})$ & 10 
& $3^31/5\,1^23/2^{5}$ & $N_{50}$ & $8\,4\,2\,1/4\,3^32/2^71$ & 15 \\
8 % & 7 & 0 
& $3\,2^2/\!/\!/$ % & B 
& $\frc17,\frc17,\frc17,\frc57$ % & $(\frc2{7},\frc2{7},\frc2{7},\frc{4}{7})$ 
% & $(\frc37,\frc37,\frc{3}{7},\frc{1}{7})$ & 12 
& $7\,1^32/3^4/2^{6}$ & \cite{Kit1},$\,N_{57}$ & $4\,3/3^21/2^31$ & 7 \\
9 % & 8 & 0
& $3^22/\!/3^21^2$ % & K & $(\frc12,\frc12,\frc14,\frc14)$ &
& $\frc1{4},\frc1{4},\frc1{2},\frc{1}{2}$ % & $(0,0,\frc{1}{4},\frc{3}{4})$ & 6 
& $4\,1^2/2^{2}1^2/3^2$ & \cite{Kit1},$\,N_{18}$ &  $5\,1/3\,2\,1/3\,2\,1$ & 6  \\
10 % & 8 & 0 
& $4\,2^2/\!/3\,2^21$ % & B9 
& $\frc14,\frc13,\frc13,\frc12$ % & $(\frc1{8},\frc1{8},\frc1{24},\frc{17}{24})$ 
% & $(\frc{3}{8},\frc{3}{8},\frc{11}{24},\frc{5}{24})$ & 5 
& $4\,1/3\,1^2/2^{2}1$ & $N_{13}$ & $5\,3\,2/4^21^2/3^22\,1^2$ & 10 \\
11 % & 8 & 0 
& $3\,2^21/\!/5\,3$ % & I24 
& $\frc15,\frc15,\frc25,\frc12$ % & $(\frc1{4},\frc1{20},\frc1{20},\frc{13}{20})$ 
% & $(\frc{1}{4},\frc{9}{20},\frc{9}{30},\frc{3}{20})$ & 9 
& $5\,1^22/2^{4}1/3^3$ & $N_{43}$ & $7\,6\,3\,2/3^52\,1/3^22^51^2$ & 18 \\
12 & idem % & I25 
& $\frc15,\frc25,\frc25,\frc12$ % & $(\frc1{4},\frc3{20},\frc3{20},\frc{19}{20})$ 
% & $(\frc{1}{4},\frc{7}{20},\frc{7}{20},\frc{11}{20})$ & 15 
& $5^21\,2^2/2^{7}1/3^5$ & $d=15$ & $7^36\,5\,3\,1/3^{11}2\,1/3^22^{14}1^2$ & 36 \\
15 % & 10 & 0 & 
& $3^22^2/\!/\!/$ % & I28 & $(\frc12,\frc12,\frc15,\frc35)$ &
& $\frc15,\frc25,\frc12,\frc12$ % $(\frc15,\frc15,\frc1{10},\frc9{10})$ 
% & $(\frc25,\frc25,\frc{3}{10},\frc{7}{10})$ & 18 
& $5^31\,2/2^{8}1^2/3^6$ & $d=18$ & $7^26\,5^73\,2/3^{18}2^3/3^52^{21}1^3$ & 60 \\
16 % & 10 & 0 & 
& $5\,3\,1^2/\!/\!/$ % & DM-32 
& $\frc25,\frc25,\frc25,\frc25$ 
% & $(\frc35,\frc35,\frc35,\frc35)$ & $(0,0,0,\frc{4}{5})$ & 24 
& $5^32^33/3^8/2^{12}\!$ & $d=24$ & $7\,5^23/3^61^2/2^{10}$ & 20 \\
17 & idem % DM-31 
& $\frc15,\frc15,\frc15,\frc15$ 
% $(\frc45,\frc45,\frc45,\frc45)$ & $(0,0,0,\frc{2}{5})$ & 12 
& $5\,1^34/3^4/2^{6}$ & \cite{Kit1},$\,N_{59}$ & $5\,3\,2/3^31/2^5$ & 10 \\
18 % & 10 & 0 
& $5\,2^21/\!/\!/$ % & I29 
& $\frc13,\frc13,\frc13,\frc45$  % $(\frc7{30},\frc7{30},\frc7{30},\frc{9}{10})$ 
% & $(\frc{13}{30},\frc{13}{30},\frc{13}{30},\frc{9}{10})$ & 6 
& $3\,1^3/5\,1/2^{3}$ & $N_{21}$ & $4\,1/3\,2/2^21$ & 5 \\
19 & idem % & I30 
& $\frc13,\frc13,\frc13,\frc25$ % & $(\frc1{30},\frc1{30},\frc1{30},\frc{7}{10})$ 
% & $(\frc{11}{30},\frc{11}{30},\frc{11}{30},\frc{3}{10})$ & 18 
& $3^51^3/5^33/2^{9}$ & $d=18$ & $8\,5\,2/4\,3^31^2/2^71$ & 15 \\ 
21 % & 12 & 0 
& $4^22^2\!/\!/3^22^21^2$ % & B11 & $(\frc12,\frc12,\frc13,\frc13)$ 
& $\frc13,\frc13,\frc12,\frc12$ %
% $(\frc1{3},\frc1{3},\frc1{2},\frc{1}{2})$ & $(0,0,\frc1{6},\frc5{6})$ & 8 
& $3^21^2/2^31^2/4^2$ & $d=12$ & $4^23\,1/4^23\,1/3^22^21^2$ & 12 \\
25 % & 12 & 0 
& $5\,3\,2^2\!/\!/3^22^21^2\!\!$ % & I33  
& $\frc13,\frc15,\frc25,\frc12$ % & $(\frc7{60},\frc1{60},\frc{11}{60},\frc{43}{60})$ 
% & $(\frc{23}{60},\frc{29}{60},\frc{19}{60},\frc{47}{60})$ & 13 
& $3^41/5^21\,2/2^61\!\!$ & $d=13$ 
& $\hspace{-7pt}7^46^35^64\,3\,1/4^43^{20}2^4\!/3^62^{30}1^6\!\!$ & 84 \\
30 % & 16 & 0 & 
& $3^42^2/\!/\!/$ % & B13 & $(\frc12,\frc12,\frc12,\frc14)$ & 9
& $\frc14,\frc12,\frc12,\frc12$ & $4^21/2^31^3/3^3$ & $d=9$ 
& $5\,4^22\,1/3^42^2/3^42\,1^2$ & 16 \\
&& $\frc1{8},\frc1{8},\frc1{8},\frc{7}{8}$ % & $(\frc38,\frc38,\frc38,\frc58)$ & 12 
& $8\,1^4/3^4/2^6$ & \cite{Kit2},$\,N_{56}$ &  $3\,1/3\,1/2^2$ & 4 \\
\hline
13 % & 9 & 1 
& $5\,3\,1/\!/\!/$ % & I27 
& $\frc25,\frc25,\frc25,\frc23$ %$(\frc2{15},\frc2{15},\frc{2}{15},\frc{14}{15})$ 
% & $(\frc7{15},\frc7{15},\frc{7}{15},\frc{11}{15})$ & 16 
& $5^22^3/3^51/2^{8}$ & $d=16$ & $7\,3\,2/3^32\,1/2^6$ & 12 \\
14 & idem % K,I26 
& $\frc15,\frc15,\frc15,\frc13$ % & $(\frc1{15},\frc1{15},\frc1{15},\frc{7}{15})$ 
% & $(\frc{4}{15},\frc4{15},\frc{4}{15},\frc{2}{15})$ & 8 
& $5\,1^3/3^22/2^{4}$ & \cite{Kit1},$\,N_{37}\!$ & $5\,1/3\,2\,1/2^3$ & 6 \\
20 % & 12 & 1 & 
& $4^22^2/\!/3^4$ % & B12 
& $\frc12,\frc13,\frc12,\frc12$ % & $(\frc1{12},\frc1{12},\frc{1}{12},\frc{11}{12})$ 
% & $(\frc5{12},\frc5{12},\frc{5}{12},\frc{7}{12})$ & 10 
& $4^22/3^31/2^{4}1^2$ & $d=10$ & $6^34^33^2/4^23^72^21^3/3^62^81^2$ & 36 \\
22 % & 12 & 1 
& $5^22/\!/3^22^21^2$ % & I36 
& $\frc13,\frc13,\frc15,\frc25$ % & $(\frc1{10},\frc1{10},\frc{1}{30},\frc{19}{30})$ 
% & $(\frc3{10},\frc3{10},\frc{13}{30},\frc{7}{30})$ & 14 
& $3^41^2/5^21\,3/2^{7}$ & $d=14$ & $8^36^25\,4\,2\,1/4^43^{10}1^2\!/2^{23}1^2$ & 48 \\
23 % & 12 & 1 
& $5\,3\,2^2/\!/5\,3^21$ % & I34 & 
& $\frc15,\frc15,\frc13,\frc12$ % & $(\frc1{12},\frc1{12},\frc{13}{60},\frc{37}{60})$ 
% & $(\frc5{12},\frc5{12},\frc{17}{60},\frc{7}{60})$ & 7 
& $5\,1^2/3^21/2^{3}1$ & $N_{27}$ & $6^32\,1/4^23^32\,1^2/3^32^51^2$ & 21 \\
24 & idem % I35 
& $\frc25,\frc25,\frc13,\frc12$ % & $(\frc1{12},\frc1{12},\frc1{60},\frc{49}{60})$ 
% & $(\frc{5}{12},\frc5{12},\frc{29}{60},\frc{19}{60})$ & 19 
& $5^32^2/3^61/2^{9}1$ & $d=19$ 
& $\!7^45^44\,3\,2/4^23^{15}2\,1^2\!/3^32^{23}1^2\!$ & 57 \\
26 % & 15 & 1 
& $5\,3^22^2/\!/\!/$ % & I38 
& $\frc13,\frc13,\frc13,\frc35$ % & $(\frc2{15},\frc2{15},\frc2{15},\frc45)$ 
% & $(\frc7{15},\frc7{15},\frc7{15},\frc{1}{5})$ & 12 
& $3^31^3/5^22/2^{6}$ &  \cite{vk09}  & $7\,5\,3/4^23\,2\,1^2/2^71$ &15 \\
27 & idem % I37 % & $(\frc{4}{15},\frc4{15},\frc{4}{15},\frc2{5})$ 
% ALTERNATIVE 
& $\frc13,\frc13,\frc13,\frc15$ & $3^71^3/5^44/2^{12}$ 
& $d=24$ & $9^25^52/4^33^{10}1^3/2^{22}1$ & 45 \\
% LOWER DEGREE 
% & $\frc15,\frc13,\frc13,\frc13$ & $5^31/3^421^22/2^{8}$ & $d=16$
% & $6^25^{11}4\,2^2\!/5^54^53^82^21^2\!/2^{36}1^3$ & 75 \\
28 % & 15 & 1 
& $5^23\,2/\!/5\,3^21^4$ % & I40 
& $\frc13,\frc13,\frc25,\frc25$ % & $(\frc25,\frc25,\frc13,\frc13)$ 
% & $(0,0,\frc{1}{15},\frc{11}{15})$ & 20 
& $3^61^2\!/5^32\,3/2^{10}\!$ & $d=20$ & $8^27\,6\,5^74^23/4^53^{17}2\,1^2/2^{37}1$
& 75 \\
29 & idem % I39 & $(\frc13,\frc13,\frc15,\frc15)$ & 
& $\frc15,\frc15,\frc13,\frc13$ % & $(0,0,\frc{2}{15},\frc8{15})$ & 12 
& $5^21^2\!/3^31\,2/2^{6}$ & $d=12$ & $6^45^34\,2/5^33^82^21^2/2^{22}1$ & 45 \\
31 % & 18 & 1 & 
& $5^23^21^2/\!/\!/$ % & DM-41 & 
& $\frc13,\frc13,\frc13,\frc13$ 
% & $(\frc2{3},\frc2{3},\frc2{3},\frc{2}{3})$ & $(0,0,0,\frc2{3})$ & 20 
& $3^51^32/5^4/2^{10}$ & \cite{vk09} & $5^53\,2/5\,4^23^51^2/2^{15}$ & 30 \\
% 32 & 18 & 1 & $7\,3^22^21/\!/\!/$ & B & $(\frc37,\frc37,\frc37,\frc23)$ 
% & $(\frc5{42},\frc5{42},\frc5{42},\frc{41}{42})$ 
% & $(\frc{19}{42},\frc{19}{42},\frc{19}{42},\frc{29}{42})$ \\
33 & $7\,3^22^21/\!/\!/$ % & & K 
& $\frc17,\frc17,\frc17,\frc23$ % & $(\frc{11}{42},\frc{11}{42},\frc{11}{42},\frc{23}{42})$ 
% & $(\frc{17}{42},\frc{17}{42},\frc{17}{42},\frc{5}{42})$ & 10 
& $7\,1^3/3^31/2^5$ & \cite{Kit2},$\,N_{48}\!$ & $8\,6\,1/4\,3^31^2/2^71$ & 15 \\
% 34 & & & & B & $(\frc27,\frc27,\frc27,\frc13)$
% & $(\frc1{42},\frc1{42},\frc1{42},\frc{25}{42})$ 
% & $(\frc{13}{42},\frc{13}{42},\frc{13}{42},\frc{11}{42})$ \\
35 % & 20 & 1 
& $5^23^22^2\!/\!/5^23^21^4\!$ % & I45 & $(\frc12,\frc12,\frc25,\frc25)$ &
& $\frc2{5},\frc2{5},\frc1{2},\frc{1}{2}$ % & $(0,0,\frc1{10},\frc9{10})$ & 24 
& $5^42^2\!/2^{11}1^2\!/3^8$ & $d=24$ & 
$7^45^54\,3/3^{19}2\,1/3^32^{24}1^3$ & 60 \\
36 & idem % I44 & $(\frc12,\frc12,\frc15,\frc15)$ &
& $\frc1{5},\frc1{5},\frc1{2},\frc{1}{2}$ % & $(0,0,\frc3{10},\frc7{10})$ & 12 
& $5^21^2\!/2^51^2\!/3^4$ & $d=12$ & $6^35^22/3^92\,1/3^32^91^3$ & 30 \\
37 % & 20 & 1 
& $5^23^22^2\!/\!/5\,3^22^41\hspace{-6pt}$ % & I43 
& $\frc25,\frc13,\frc13,\frc12$ % & $(\frc1{20},\frc1{20},\frc7{60},\frc{47}{60})$ 
% & $(\frc9{20},\frc9{20},\frc{23}{60},\frc{17}{60})$ & 17 
& $5^32/3^51^2\!/2^81$ & $d=17$ 
& $\!\!\!7^46\,5^94\,2/4^63^{18}2^21^3\!/3^52^{33}1^4\!\!\!$ & 85 \\
38 & idem % & & & I42 
& $\frc15,\frc13,\frc13,\frc12$ % & $(\frc3{20},\frc3{20},\frc1{60},\frc{41}{60})$ 
% & $(\frc7{20},\frc7{20},\frc{29}{60},\frc{11}{60})$  & 11 & d=55
& $5^21/3^31^2\!/2^51$ &  \cite{vk09} & $6^35^64\,3/4^63^82^21^3/3^52^{18}1^4$ & 55 \\
39 % & 24 & 1 
& $5^23^22^4/\!/\!/$ % & I46 
& $\frc13,\frc13,\frc13,\frc12$ % & $(\frc1{12},\frc1{12},\frc1{12},\frc{3}{4})$ 
% & $(\frc5{12},\frc5{12},\frc{5}{12},\frc{1}{4})$ & 15 
& $3^41^3/2^71/5^3$ & $d=15$ & $5^53\,2/4^43^41^2/3^22^{11}1^2$ & 30  \\
\hline
\end{tabular} \label{tb:abmaps} \vspace{-8pt}
\caption{AB-maps for algebraic Painlev\'e VI solutions of genus $0$ (in the upper part)
or genus 1 (in the lower part).}
\end{table}

With the construction of Theorem \ref{th:kitaev} in mind,
% construction of Painlev\'e VI solutions by  in mind,
we computed AB-maps for all Lisovyy-Tykhyy cases
algebraic Painlev\'e VI solutions of genus 0 or 1.
The results are presented in Table \ref{tb:abmaps}, with the genus 0 and 1 cases
separated by a horizontal line. The first column gives the enumeration in \cite{LiTy}.

The second column of Table \ref{tb:abmaps}
 gives the branches permutation monodromy of the Painlev\'e VI solutions,
using the fact that in a parametrization $(q(s),t(s))$ of those algebraic solutions,
$t(s)$ is a Belyi map (by the Painlev\'e property). The second column gives
the passport of that Belyi map (without the [\,] delimiters), 
but the notation is compacted when branching patterns
in 2 or all 3 fibers is the same. The repetition is indicated by the number of $/$'s.
For example, $3^2\!/\!/2^21^2$ for the solution \#3 means the passport $[3^2\!/3^2\!/2^21^2]$,
and $3^22^2\!/\!/\!/$ for the solution \#15 means the passport $[3^22^2\!/3^22^2\!/3^22^2]$,
etc. The algebraic degree of the Painlev\'e VI solution can be quickly determined
from the passport. 

The third column gives the exponent differences of representative Painlev\'e VI equations
$P_{VI}(\theta_0,\theta_1,\theta_t,\theta_\infty)$. 
Two distinct Painlev\'e VI equations are given for the parametric solution IV,
because they are generally not related by Shlessinger and fractional-linear transformations,
and AB-maps (of degree 3 and 6) exist for both of them.
The case \#30 is represented by two Painlev\'e VI equations solutions for the same reason,
while \#3, \#5 take two lines each because two AB-maps for them are already known.

The fourth column gives the passport of a $(2,3,m)$-minus-4 regular AB-map 
giving an algebraic solution of $P_{VI}(\theta_0,\theta_1,\theta_t,\theta_\infty)$
by Theorem \ref{th:kitaev}. The three fibers are ordered to match the % corresponding
order of the $\theta_j$'s in the third column conveniently. 
The fifth column either gives the degree $d$ of the AB-map if it was not computed  previously,
or gives references to \cite[Table 1]{vHK5} (by the $N_j$-label)
and other publications \cite{AK}, \cite{Kit1}, \cite{Kit2}, \cite{VK2}.
Given $\theta_0>0,\theta_1>0,\theta_t>0,\theta_\infty<1$,
the degree of the pull-back map from $E(1/2,1/3,1/m)$ equals
\begin{equation}
d=\frac{\theta_0+\theta_1+\theta_t-\theta_\infty}{\frac12+\frac13+\frac1m-1}.
\end{equation}
This follows from the Hurwitz theorem, 
or (assuming the AB-map is defined over $\RR$) 
by geometric consideration of spherical or hyperbolic areas 
in analytic continuation of pulled-back hypergeometric functions
by the Schwarz reflection principle \cite[Lemma 6.2, etc.]{VK2}.

%or gives the degree of the AB-map if it was not computed  previously.

The last two columns % give the passport and the degree of 
characterize an important Belyi map derived from each AB-map $\varphi(x,\w)$. 
All presented AB-maps are parametrized 
(as Hurwitz spaces of dimension 1) by algebraic curves of genus 0,
with $\w$ as a minimal projective parameter of those curves. 
The fourth fiber $\psi(w)=\varphi(q,\w)$ of the extra branching point $x=q$ 
is a function of $\w$ that is intrinsic to $\varphi(x,\w)$. It gives the braid group action
on $\varphi(x,\w)$ as the fourth fiber is moved continuously around the other three fibers. 
The  function $\psi(w)$ is a Belyi map  \cite[Remark 5.3]{vHK5}, 
and is a good measure of complexity of the AB-map.
The passport and degree $d^*$ of $\psi(w)$ are given
in the last two columns of Table \ref{tb:abmaps}.
For the $\w$-values in the three critical fibers of $\psi(\w)\in\{0,1,\infty\}$,
the AB-map specializes to Belyi maps of degree $\le d$.

\begin{remark} \rm \label{rm:commonab}
The cases \#32, \#34 are skipped in Table \ref{tb:abmaps},
because Schlessinger transfomations are necessary 
to obtain those  Painlev\'e VI solutions. As we discussed in Example \ref{ex:kitaevshift},
the AB-map of \#33 has to be applied for a  pull-back from $E(1/2,1/3,2/7)$
or $E(1/2,1/3,3/7)$. Kitaev \cite{Kit1} stresses that
pairs of icosahedral cases with the same monodromy 
(such as \#6, \#7; see the second column in Table \ref{tb:abmaps})
can be similarly obtained by pull-backs with respect to a common AB-map 
applied to $E(1/2,1/3,1/5)$  and $E(1/2,1/3,2/5)$, 
with a Schlessinger transformation necessary after one or other pull-back.
\end{remark}

Examples of AB-maps for the solutions \#40\,--\,\#45 in \cite{LiTy} of genus 2, 3 or 7 
remain to be computed. But even these cases can be considered as handled
if we allow Kitaev's quadratic transformations \cite{Kit91} of Painlev\'e VI solutions
and corresponding isomonodromic Fuchsian systems. Derivation of the Painlev\'e VI solutions
by these quadratic transformations is demonstrated in \cite{KV1}.

% Almost Belyi maps for Lisovyy-Tykhyy solutions \#18, \#8, \#33, \#3
% have been already considered here in Examples 
% \ref{eq:d6example}, \ref{ex:lt08}, \ref{ex:lt33}, \ref{ex:lt03}, respectively. 

The AB-maps presented in Table \ref{tb:abmaps} 
are not necessarily unique  for the passports given in fourth column. 
For example, \cite[Table 1]{vHK5} gives also composite maps 
with the degree 6 and 12 passports of the entries \#9 and \#30.
Another composite map with the degree 20 passport for % the entry 
\#31 is given in \cite[\S 5]{VK2}. 
As explained in \cite{Kit1}, compositions of Belyi maps 
with an AB-map $\varphi_0$ give Painlev\'e VI solutions 
(by the RS-transformations) that can be obtained from $\varphi_0$ already.
Thus composite AB-maps are not useful in deriving complicated
algebraic Painlev\'e VI solutions.
% give specializations of the parametric solutions \#II\,--\,\#IV or other 
% algebraic Painlev\'e VI solutions that can be obtained from

\subsection{Computation of AB-maps}

Here we demonstrate computation of AB-maps for the Painlev\'e VI solutions 
\#15 and \#22. % and \#39.
As these examples show, % following examples demonstrate, 
identification of Fuchsian equations (\ref{eq:ode4p1}), (\ref{eq:qpback})
using Painlev\'e VI solutions straightforwardly gives 
the singularity polynomials $F,G,H$, \mbox{$x-q$} of the AB-maps
(and accessory parameters of Fuchsian equations)
and a ready, convenient parameter of the Hurwitz curve. 

\begin{example} \rm
To find an AB-map with the passport $[3^6\!/5^32\,1/2^{8}1^2]$
for the algebraic solution \#15, we are looking for a polynomial identity
\begin{equation} \label{eq:abc15}
P^3+r_0 Q^5G=R^2H
\end{equation}
with $P=x^6+a_1x^5+\ldots+a_6$, $Q=x^3+b_1x^2+b_2x+b_3$, 
$R=x^8+c_1x^7+\ldots+c_8$, $G=x$ and $H=x^2+d_1x+d_2$. 
After clearing denominators in the logarithmic derivative ansatz 
(\ref{eq:logdif1})--(\ref{eq:logdif9}) with $h_1=h_2=2$, $S=GH/(x-q)$,
we get the equations
\begin{align*}
0= & \, (2q+7b_1-a_1-2c_1)\,x^8+(12b_2\!-\!4a_2\!-\!2c_2\!+\!2qc_1\!+\!4a_1b_1)\,x^7+\ldots \\
0= & \, (2q+7b_1-4a_1+d_1)x^{12}\!+\!(
12b_2\!-\!4a_2\!-\!2c_2\!+\!4qa_1\!+\!5b_1c_1 %\!-\!2a_1^2\!-\!c_1d_1\!+\!6b_1d_1
+\ldots)\,x^{11}\!+\ldots.
\end{align*}
From their leading coefficients we can consequently eliminate all coefficients of $P$, $R$
except $a_2$. Next we compute the pull-back (\ref{eq:qpback})--(\ref{eq:qpback9}),
with $k=3$, $\ell=2$, $m=5$, thus $a=-1/60$, $b=11/60$.
The coefficient $W_2$ % to $Y(x)$ 
in (\ref{eq:qpback}) equals
\begin{equation*}  \label{eq:pbackyx} \!
% W_2\!=\!
\frac{ % 108x^9\!+\!4(11a_1\!+\!82b_1\!-\!104d_1\!-\!289q)x^8
% \!+\!(44a_2\!+\!1128b_2\!-\!896d_2\!+\!132q^2
27x^9\!+\!(11a_1\!+\!82b_1\!-\!104d_1\!-\!289q)x^8
\!+\!(11a_2\!+\!282b_2\!-\!224d_2\!-\!\frac{11}{4}b_1^2 %\!+\!33q^2
\!+\!\ldots)x^7\!+\!\ldots}{900% 3600
\,(q-x)\,H\,G^2\,Q^2}.
\end{equation*} 
To compute the corresponding equation (\ref{eq:ode4p1}),
we start with this solution $q_{15}(t_{15})$ of $P_{VI}(1/5,1/2,1/2,3/5)$:
\begin{align}
q_{15}= & -\frac{2s(s-1)(s-5)^2(s^2-3)(s^2+4s+5)}
{(s+1)^2(s+5)(s^2-4s+5)(s^4+6s^2-75)},  \\
t_{15}= & -\frac{(s-1)^3(s-5)^3(s^2+4s+5)^2}{(s+1)^3(s+5)^3(s^2-4s+5)^2}.
\end{align}
It differs from the solution of $P_{VI}(1/2,1/5,1/2,2/5)$ in \cite{LiTy}
by the fractional-linear transformation $(q_{15},t_{15})\mapsto (1-q_{15},1-t_{15})$.
We express the entities in (\ref{eq:hamilt})--(\ref{eq:ode4p1})
in the parametrized form:
\[
p_{15}=-\frac{s(s+1)^2(s+5)(s^2-4s+5)(s^4+6s^2-75)}
{10(s-1)(s-5)^2(s^4-25)(s^2+4s+5)}, \quad
\Theta = -\frac{3}{100}, \quad\mbox{etc.}
\]
The symmetry between $x=1$ and $x=t_{15}$ is realized by $s\mapsto -s$.
To identify \mbox{$(x-1)(x-t_{15})$} with the irreducible polynomial $H$,
we scale $x\mapsto x/K$ with 
\[
K=s\,(s+1)^3(s+5)^3(s^2-4s+5)^2.
\] 
The coefficient $W_1$ % to $Y(x)$ 
in  (\ref{eq:ode4p1}) is thereby
divided by $K^2$ (along with the substitution of $x$) 
and becomes a function of the invariant $u=s^2$:
\begin{equation*}
W_1=\frac{3x^2+\frac{6u(41u^6-900u^5 %+11129u^4
+\ldots+46875)}{u^2+6u-75}x+
\frac{4u^2(u-1)(u-3)(u-25)^2(u^2-6u+25)(5u^5+\ldots-9375)}{u^2+6u-75}}
{100\,(q-x)\,G\,H},
\end{equation*}
with explicitly
\begin{align}
H= &\, x^2-4u(5u^4\!-\!80u^3\!+\!678u^2\!-\!2000u\!+\!3125)x % \nonumber\\ & 
    -u(u\!-\!1)^3(u\!-\!25)^3(u^2\!-\!6u\!+\!25)^2, \nonumber \\
q=  &-\frac{2u(u-1)(u-3)(u-25)^2(u^2-6u+25)}{u^2+6u-75}. % \nonumber
\end{align}
This parametrizes $d_1,d_2,q$. 
The remaining coefficients $a_2,b_1,b_2,b_3$ are obtained
from the identification \mbox{$W_1=W_2$}. 
After clearing denominators, we get a polynomial expression of degree 8 in $x$.
The leading coefficients gives immediately
\begin{equation}
b_1=-\frac{8u(u^6-15u^5-14u^4+3326u^3-29575u^2+100625u-187500)}{u^2+6u-75}.
\end{equation}
The coefficient to $x^7$ is linear in $a_2,b_2$, and the next two coefficients 
are linear in $b_3$. After elimination of $b_2,b_3$, 
we get a quadratic polynomial in $a_2$ that factorizes.
We check both candidates for $a_2$ on another equation,
and the correct value is
\begin{align*}
a_2= & -4u(u^{10}+1340u^8-38600u^7+421150u^6-3081320u^5+20032500u^4 \nonumber \\
& \qquad\; -97975000u^3+131015625u^2+703125000u-2109375000).
\end{align*}
This gives
\begin{align*}
b_2= & -\frac{64u(u-25)^3(11u^6-165u^5+968u^4-3082u^3+6875u^2-20625u+31250)}{u^2+6u-75},\\
b_3= &\, \frac{512u^2(u-3)(u-25)^6(u^2-6u+25)^2}{u^2+6u-75}
\end{align*}
and the other coefficients. The factor $r_0$ can be determined 
by dividing the left-hand side of (\ref{eq:abc15}) by $H$ with respect to $x$,
and looking at the remainder.
We find $r_0=27u(u^2+6u-75)^5$.

Simplification of the obtained AB-map to a presentable size is a tedious,
less automated task that may take much more time than the above computation.
The basic ideas are to simplify the Belyi map $\varphi(q(u),u)$ stated in
the last two columns in Table \ref{tb:abmaps} (of degree 60);
simplification of elliptic surfaces such as $y^2=GQ$;
and considering factorization of the discriminants,
resultants of $P,Q,R,H$ with respect to $x$.
For example, the transformation $u=5v$, \mbox{$x=100x+500v(v-5)^2(5v^2-6v+5)$}
is useful for a start, introducing high powers of $(v-1)$ in the coefficients
while keeping the powers of {$v,v-5,5v^2-6v+5$}.
\end{example}

\begin{example} \rm
To find an AB-map with the passport $[3^41^2/5^21\,3/2^{7}]$
for the Painlev\'e VI solution \#22, we are looking for a polynomial identity
\begin{equation} \label{eq:abc22}
P^3F+r_0 Q^5G=R^2
\end{equation}
with $P=x^4+a_1x^3+\ldots+a_4$, $Q=x^2+b_1x+b_2$, 
$R=x^7+c_1x^6+\ldots+c_7$, $F=x^2+d_1x+d_2$ and $G=x+e_1$.
We do not hurry with setting $e_1=0$ by choosing the point $x=0$. 
In the logarithmic derivative ansatz we have $h_1=h_2=3$, $S=FG/(x-q)$.
It allows to eliminate straightforwardly  all coefficients of $P,R$ except $a_3$. 
We calculate the coefficient $W_2$ % to $Y(x)$ 
in (\ref{eq:qpback}). % and denote it by $U_1$ analogously to (\ref{eq:pbackyx}).

To compute the Fuchsian equation (\ref{eq:ode4p1}),
we use the Painlev\'e VI solution of $P_{VI}(1/3,1/3,1/5,2/5)$ from \cite{LiTy},
with $z=\sqrt{3(5s+1)(8s^2-9s+3)}$:
\begin{align}
q_{22}= & \, \frac12+\frac{140s^6+1029s^5-1023s^4+360s^3-288s^2+27s+27}
{18z(s+1)(7s^3-3s^2-s+1)},  \\
t_{22} = & \, \frac12+\frac{40s^6+540s^5-765s^4+540s^3-270s^2+27}{6z(s+1)^2(8s^2-9s+3)}.
\end{align}
We wish to utilize the symmetry $z\mapsto -z$, $(q_{22},t_{22})\mapsto(1-q_{22},1-t_{22})$
while identifying $FG$ with $x(x-1)(x-t_{22})$.
For this purpose we find an elliptic surface that is defined over $\QQ(t(1-t))$
and has the same $j$-invariant as the Legendre family 
\mbox{$y^2=x(x-1)(x-t)$}. 
The following elliptic surface has these properties:
\begin{equation} \label{eq:altleg}
y^2=(x-t)\,(x-1+t)\,\big(x-2t(1-t)\big).
\end{equation}
Therefore we identify
\begin{align}
F= &\; (x-t_{22})\,(x-1+t_{22})=x^2-x+t_{22}\,(1-t_{22}), \\ 
G= &\; x-2\,t_{22}\,(1-t_{22}) \nonumber
\end{align}
initially. Here $t_{22}\,(1-t_{22})$ % \in\QQ(s)$; that is, it 
is not dependent on $z$: %\subset\QQ(s,z)$:
\begin{equation}
t_{22}\,(1-t_{22})=-\frac{16s^5(s-3)^5(5s-3)^2}{27(s+1)^4(5s+1)(8s^2-9s+3)^3}.
\end{equation}
Additionally, we transform 
\begin{equation} \label{eq:nextx}
x\mapsto \frac12
+\frac{40s^6+540s^5-765s^4+540s^3-270s^2+27}{54(s+1)^4(5s+1)(8s^2-9s+3)^3}\,x
\end{equation}
to get the simpler
\begin{align}
F= &\; x^2-27(s+1)^4(5s+1)(8s^2-9s+3)^3, \\ 
G= &\; x+40s^6+540s^5-765s^4+540s^3-270s^2+27. \nonumber
\end{align}
This parametrizes $d_1,d_2,e_1$.
An isomorphism from the Legendre curve to (\ref{eq:altleg})
is given by $x\mapsto (x-t)/(1-2t)$. 
The composition of this isomorphism (with $t=t_{22}$) 
and (\ref{eq:nextx}) is the transformation
\begin{equation}
x\mapsto Kx+\frac12, \quad \mbox{with} \quad
K=-\frac{z}{18(s+1)^2(5s+1)(8s^2-9s+3)^2}.
\end{equation}
After this whole transformation, the coefficient $W_1$ % to $Y(x)$ 
in (\ref{eq:ode4p1}) equals
\begin{equation*}
W_1=\frac{77x^2+\frac{8(30625s^{9}+\ldots-2673s+2673)}{3(7s^3-3s^2-s+1)}x
+\frac{(s+1)(8s^2-9s+3)(666400s^{12}+\ldots+136323)}{3(7s^3-3s^2-s+1)}}
{900(q-x)FG}
\end{equation*}
with
\begin{equation*}
q=-\frac{(s+1)(8s^2\!-\!9s\!+\!3)
(140s^6\!+\!1029s^5\!-\!1023s^4\!+\!360s^3\!-\!288s^2\!+\!27s\!+\!27)}
{3(7s^3-3s^2-s+1)}.
\end{equation*}
The identification $W_1=W_2$ leads to a polynomial of degree 6 in $x$
after clearing the denominators. Its 3 leading coefficients give straightforwardly
\begin{align*}
b_1=&\; \frac{2(8s^2-9s+3)^2(16s^4-8s^3+8s^2+15s+3)}{7s^3-3s^2-s+1}, \\
b_2=& -\frac{(s+1)^2(8s^2\!-\!9s\!+\!3)^3
(625s^6\!+\!1386s^5\!-\!567s^4\!+\!540s^3\!-\!27s^2\!-\!162s\!-\!27)}{7s^3-3s^2-s+1}, \\
a_3=& -2(8s^2-9s+3)^3
(192500s^{10}+300697s^9+68513s^8+41532s^7+297588s^6\\
&\hspace{60pt} -86778s^5+57510s^4+43740s^3-19440s^2-10935s-1215).
\end{align*}
The logarithmic derivative ansatz already gave expressions of the other coefficients
in terms of $a_3,b_1,b_2,d_1,d_2,e_1,q$. With all coefficients parametrized,
we find $r_0=13824(5s+1)(7s^3-3s^2-s+1)^5$.
\end{example}

% 15, g=0, d=18: 5^3 2 1 / 2^8 1^2 / 3^6
% 20, g=1, d=10: 4^2 2 / 3^3 1 / 2^4 1^2
% 21, g=0, d=12: 3^2 1^2 / 2^3 1^2 / 4^2
% 22, g=1, d=14: 3^4 1^2 / 5^2 3 1 / 2^7
% 25, g=0, d=13: 3^3 1 / 5^2 2 1 / 2^6 1
% 29, g=1, d=12: 5^2 1^2 / 3^3 2 1 / 3^6
% 36, g=1, d=12: 5^2 1^2 / 2^5 1^2 / 3^4
% 39, g=1, d=15: 3^4 1^3 / 2^7 1 / 5^3
% 40, g=2, d=16: 
% 42, g=3, d=20:
% 43, g=3, d=21:
% 45, g=7, d=25:

\end{document}